\documentclass[11pt,reqno]{amsart}

\usepackage[paper,pkg/amsaddr=true,pkg/biblatex/opt={+giveninits=true,maxalphanames=6},thm/numberwithin={}]{zbs}
\usepackage{xurl}

\NewEnvironmentCopy{theorem}{thm}
\NewEnvironmentCopy{proposition}{prop}
\NewEnvironmentCopy{lemma}{lem}
\NewEnvironmentCopy{corollary}{cor}

\newcommand{\R}{\mathbb{R}}

\newcommand{\Q}{\mathbb{Q}}
\newcommand{\Z}{\mathbb{Z}}
\newcommand{\F}{\mathbb{F}}
\newcommand{\OK}{\mathcal{O}_K}
\newcommand{\Norm}{\mathrm{N}}
\newcommand{\Disc}{\Delta}
\newcommand{\cP}{\mathcal{P}}
\newcommand{\fp}{\mathfrak{p}}
\newcommand{\fa}{\mathfrak{a}}

\DeclareMathOperator{\rd}{rd}
\DeclareMathOperator{\Gal}{Gal}

\newcommand{\aiso}{\alpha_{\mathrm{iso}}}
\newcommand{\adist}{\alpha_{\mathrm{dist}}}
\newcommand{\fiso}{f_{\mathrm{iso}}}
\newcommand{\fdist}{f_{\mathrm{dist}}}

\title{The Minkowski grid has robustly many repeated distances}
\author[Sungchul Lee]{Sungchul Lee\nfts{1}}
\address{\nfts{1}Department of Mathematical Sciences, Seoul National University, Republic of Korea.}
\email{lsngchl127@snu.ac.kr}
\author[Cosmin Pohoata]{Cosmin Pohoata\nfts{2}}
\address{\nfts{2}Department of Mathematics, Emory University, United States.}
\email{cosmin.pohoata@emory.edu}
\author[Daniel G.\ Zhu]{Daniel G.\ Zhu\nfts{3}}
\address{\nfts{3}Department of Mathematics, Princeton University, United States.}
\email{zhd@princeton.edu}
\date{}
\thanks{S.L.\ was supported by the National Research Foundation of Korea (RS-2024-00342160). C.P.\ was supported by NSF grant DMS-2246659. D.Z.\ was supported by the NSF Graduate
Research Fellowship Program (grant DGE-2039656).}

\begin{document}

\begin{abstract}
We show that there exists a constant $\delta > 0$ such that for any positive integer $n$ there exists a set of $n$ points $P \subset \mathbb{R}^2$ with the following property: for every subset $A \subseteq P$ of size $|A| \geq 2$, 
\[ \max_{\lambda>0}
        \#\{(a,b)\in A \times A:
        a\ne b,\ \abs{a-b}=\lambda\} \gsim \frac{|A|^2}{n^{1-\delta}}.\]
Our result is a vertical amplification of a robust Ramanujan estimate recently established by Croot-Mao-Pohoata-Sheffer-Yip for arbitrary subsets of the ordinary square grid, and is inspired by recent constructions for the Erd\H{o}s unit distance problem and the Elekes-R\'onyai problem. 

Taking $A=P$, the inequality above gives a distance occurring $n^{1+\delta}$ times in $P$; thereby a scaled copy of $P$ is a counterexample for the unit-distance conjecture. In addition, the same inequality shows that
\begin{enumerate}
    \item all subsets of $P$ of size $\gsim n^{1-\delta}$ must contain isosceles triangles, and
    \item all subsets of $P$ of size $\gsim n^{1/2-\delta}$ must contain repeated distances. 
\end{enumerate}
These features give polynomially improved estimates for old problems of Erd\H{o}s. The existence of a set satisfying property (1) confirms a conjecture of Erd\H{o}s from 1980, whereas the existence of a set with property (2) answers a question of Conlon-Fox-Gasarch-Harris-Ulrich-Zbarsky in the negative. 
\end{abstract}

\maketitle

\section{Introduction}
\label{sec:introduction}

If $r_2(m)=\#\{(u,v)\in\Z^2:u^2+v^2=m\}$, then it is a classical number theoretic fact (see e.g.\ \cite{Ramanujan1915}) that
\begin{equation}
\label{eq:ramanujan-intro}
        \max_{1\le m\le x} r_2(m)
        \ge
        \exp\left(c\frac{\log x}{\log\log x}\right)
\end{equation}
for all sufficiently large \(x\). Erd\H{o}s famously observed \cite{Erdos46} that this implies that in the grid \([\sqrt{n}] \times [\sqrt{n}]\), there is a distance which occurs $\geq n\exp(c\log n/\log\log n)$ times, which led Erd\H{o}s to the well-known unit distance conjecture.

In \cite{CMPSY2026}, a robust form of \labelcref{eq:ramanujan-intro} was recently established. For a finite point set \(A\subset\R^2\), define
\[
        \mu(A)
        \coloneq
        \max_{\lambda>0}
        \#\{(a,b)\in A \times A:
        a\ne b,\ \abs{a-b}=\lambda\}.
\]
In other words, \(\mu(A)\) is the largest ordered-pair multiplicity of a nonzero distance
in \(A\). Both \cite[Theorem 1.3]{CMPSY2026} and \cite[Theorem 1.4]{CMPSY2026} rest upon the following estimate.

\begin{theorem}[Croot-Mao-Pohoata-Sheffer-Yip] \label{thm:cmpsy-grid-sieve-intro}
There exists an absolute constant \(c>0\) such that, for every
\(A\subseteq [\sqrt{n}]^2\) of size at least $2$,
\[\mu(A) \gsim \frac{\abs{A}^2}{n} \cdot \exp\left(c\frac{\log n}{\log\log n}\right).
\]
\end{theorem}
Since distances in the grid lie in the set $\sqrt{1},\sqrt{2},\ldots,\sqrt{2n}$, we immediately have $\mu(A) \gsim \abs{A}^2/n$ by pigeonhole. \cref{thm:cmpsy-grid-sieve-intro} states that for \emph{every} subset of the grid $[\sqrt{n}]^2$, not just the grid itself, this bound can be sharpened by the same Ramanujan-type factor. In \cite{CMPSY2026}, this was proven using a new combinatorial large sieve method and was used to establish new estimates for subsets of $[\sqrt{n}]^2$ without repeated distances or without isosceles triangles, improving on classical results. We refer to \cite{CMPSY2026} for more context.

Recent breakthrough work from OpenAI \cite{OpenAIUnitDistanceBlog,UnitDistanceRemarks} constructed grids of $n$ points, derived from number fields of high degree, containing a distance which occurs $\geq n^{1+\delta}$ times, disproving Erd\H{o}s's unit distance conjecture. Roughly speaking, in the language from \cite{PohoataER}, the core insight was that the one-dimensional number theoretic phenomenon behind \labelcref{eq:ramanujan-intro} can be vertically amplified by passing to arbitrary number fields. Furthermore, if the number fields are chosen carefully, it can be more efficient to increase the degree than the horizontal width of the construction. This horizontal vs.\ vertical amplification principle has since been used (implicitly or explicitly) to find improved constructions for other longstanding problems in combinatorial number theory, such as the sum-product problem over the reals \cite{BSSZ2026} and the Elekes-R\'onyai problem \cite{PohoataER}. See the latter and, more recently, also \cite{TaoBlog} for more systematic discussions about amplification. In \cite{PohoataER} and \cite{CMPSY2026}, the authors also outlined a more general heuristic for possibly generating other constructions by vertically amplifying some of the recent new results from \cite{CMPSY2026}.

The purpose of this note is to use such a strategy to sharpen \cref{thm:cmpsy-grid-sieve-intro}. Our main result is that grids constructed from number fields of high degree can serve as a robust counterexample to the unit distance conjecture:
\begin{theorem}
\label{thm:robust-ramanujan}
There exists a constant $\delta > 0$ such that for any positive integer $n$ there exists an $n$-element set \(P\subset\R^2\) with the following property: for every subset
\(A\subseteq P\) of size at least $2$, we have
\[\mu(A) \gsim \frac{\abs{A}^2}{n^{1-\delta}}.\]
\end{theorem}
(Though quantitative aspects of the construction from \cite{OpenAIUnitDistanceBlog} have been getting a lot of attention lately \cite[\href{https://teorth.github.io/optimizationproblems/constants/84a.html}{84a}]{Sawin2026,TaoOptimization84a}, we do not attempt here to optimize the value of $\delta$.)

By setting $A=P$, we find that $\mu(P)\gsim n^{1+\delta}$, so an appropriately scaled version of $P$ contains $\gsim n^{1+\delta}$ unit distances. This on its own should not be a surprise:
as in the OpenAI construction, the set $P$ from \cref{thm:robust-ramanujan} is also a two-dimensional \emph{Minkowski grid} obtained from a symmetric box in the ring of integers of a high-degree totally real number field. Curiously, our proof in this special case differs somewhat from \cite{OpenAIUnitDistanceBlog,UnitDistanceRemarks} and may be of independent interest; notably, it bypasses all discussion of the class group. Our main point, however, is that our set $P$ has a few other new interesting features.

\subsection*{Applications}
For a finite point set \(P\subset\R^2\), write
\[
        \aiso(P)
        \coloneq
        \max\{|A|:A\subseteq P,\ A\text{ contains no isosceles triangle}\}.
\]
Throughout, degenerate isosceles triangles (three equally
spaced collinear points) are also forbidden.  Equivalently,
\(\aiso(P)\) is the independence number of the 3-uniform hypergraph on \(P\)
whose edges are pairwise distinct triples of points $a,b,c\in P$ with $|a-b|=|a-c|$. Let $\fiso(n)\coloneq \min_{|P|=n}\aiso(P)$. Thus \(\fiso(n)\) is the largest integer \(m\) such that every \(n\)-point planar set contains an \(m\)-point subset with no isosceles triangle. 

Similarly, define
\[
        \adist(P)
        \coloneq
        \max\{|A|:A\subseteq P,\ A\text{ determines no repeated distance}\},
\]
where no repeated distance means that the distances determined by unordered
pairs of distinct points of \(A\) are all distinct.  Let $\fdist(n)\coloneq \min_{|P|=n}\adist(P)$.

Determining the asymptotics of $\fiso(n)$ and $\fdist(n)$ are old problems of Erd\H{o}s \cite{Erdos1980Survey,Erd57} (see also \cite[\S 5.3]{BMP}, \cite[\S 6]{Sheffer2014}, and \cite[\href{https://www.erdosproblems.com/1207}{\#1207} and \href{https://www.erdosproblems.com/1208}{\#1208}]{Bloom}). The closely related problem of determining $\aiso(P)$ and $\adist(P)$ for $P$ the $\sqrt{n} \times \sqrt{n}$ square grid has also been studied \cite{ErGu70,LefmannThiele95,ChartonEllenbergWagnerWilliamson2024}, and more recently in \cite{CMPSY2026}.

It is clear that if $A$ contains no isosceles triangle, then $\mu(A) \leq \abs{A}$, and that if $A$ determines no repeated distance, then $\mu(A) = 2$. Thus \cref{thm:cmpsy-grid-sieve-intro} readily implies the following upper bounds for $\fiso(n)$ and $\fdist(n)$, which were previously the best known:
\[\fiso(n) \lsim n \exp\left(- c \log n / \log \log n\right)\quad \text{and}\quad \fdist(n) \lsim n^{1/2} \exp\left(- c \log n / \log \log n\right).\]

\cref{thm:robust-ramanujan} gives polynomial improvements for both problems.
\begin{corollary}
\label[corollary]{cor:simultaneous}
There is an absolute constant \(\delta>0\) such that, for every sufficiently large
\(n\), there is a single set of $n$ points \(P\subset\R^2\) simultaneously satisfying all three of the following properties:
\begin{enumerate}
\item There is a positive distance attained by \(\geq n^{1+\delta}\) unordered pairs
of points of \(P\).
\item Every subset \(A\subseteq P\) with \(|A|\ge n^{1-\delta}\) contains an
isosceles triangle.
\item Every subset \(A\subseteq P\) with \(|A|\ge n^{1/2-\delta}\) determines a repeated distance.
\end{enumerate}
\end{corollary}
In particular, it follows that $\fiso(n) \leq \aiso(P) \lsim n^{1-\delta}$ and $\fdist(n)\leq \adist(P) \lsim n^{1/2-\delta}$. The first result confirms a conjecture by Erd\H{o}s from 1980 \cite[p.~110]{Erdos1980Survey}, while the second answers a question of Conlon, Fox, Gasarch, Harris, Ulrich, and Zbarsky from \cite{ConlonFoxGasarchHarrisUlrichZbarsky15}, who specifically asked whether $\fdist(n) = n^{1/2-o(1)}$. 

These results imply a polynomial separation from the one-dimensional analogues of $\fiso(n)$ and $\fdist(n)$, which are known to be $n^{1-o(1)}$ and $\Theta(n^{1/2})$. Specifically, Koml\'os, Sulyok, and Szemer\'edi \cite{KSS} showed that every $n$-element set $P \subset \mathbb{R}$ contains a subset of size $\gsim r_{3}(n)$ with no three-term arithmetic progressions and a subset of size $\gsim s(n)$ that is Sidon. Here $r_{3}(n)$ denotes the size of the largest subset of $\left\{1,2,\ldots,n\right\}$ without three-term arithmetic progressions and $s(n)$ denotes the size of the largest Sidon subset of $\left\{1,2,\ldots,n\right\}$. On the other hand, it is well-known that $r_3(n) = n^{1-o(1)}$ \cite{Beh,Roth} and $s(n) = (1+o(1))n^{1/2}$ \cite{Singer38, ErdosTuran41}.

Last but not least, we would like to note that the existence of point sets satisfying properties like (2) and (3) is philosophically related to the phenomenon from \cite{BaloghSolymosi2018}, where Balogh and Solymosi use hypergraph containers to build point sets with no four points collinear but with the property that every large subset contains a collinear triple.  The present construction again produces a geometric $3$-uniform hypergraph with polynomially small independent sets, though in this case the condition is metric and not projective.

\subsection*{Proof sketch}
The proof of \cref{thm:cmpsy-grid-sieve-intro} relies on the fact that for
\(p\equiv1\pmod4\), the congruence $x^2+y^2\equiv0\pmod p$ factors as $(x + \iota y)(x - \iota y) \equiv 0 \pmod p$, where $\iota$ is a square root of $-1$ modulo $p$. Given a set $A \subseteq [\sqrt{n}] \times [\sqrt{n}]$, consider the map $A \to \setf_p$ given by sending $(x,y) \mapsto x+\iota y$. If $a$ and $b$ map to the same element, then $\abs{a-b}^2$ is divisible by $p$, and by Cauchy-Schwarz this happens for at least $\abs{A}^2/p$ pairs $(a,b)$, where the possibility that $a = b$ can be ignored if $\abs{A}$ is much larger than $p$. If we apply similar logic to the map $(x,y) \mapsto x - \iota y$ and, for simplicity, assume that these events are effectively disjoint, the squared distance $\abs{a-b}^2$ is divisible by $p$ roughly $2/p$ of the time, instead of the na\"ive guess of $1/p$. By generalizing this argument to multiple $1$ mod $4$ primes $p_1,p_2,\ldots,p_k$, one can show that $\abs{a-b}^2$ is divisible by $Q \coloneq p_1p_2\cdots p_k$ roughly $2^k/Q$ of the time, which suggests a bound of $\mu(A) \gsim 2^k \abs{A}^2/n$ after pigeonholing. However, for us to ignore the cases where $a = b$, we need $\abs{A}$ to be at least on the order of $Q$, which forces $k \asymp \log n/\log \log n$.

Here, we instead fix some (large) number of $1$ mod $4$ primes $p_1,p_2,\ldots,p_k$, and consider grids over number fields $K$ of growing degree $d$ such that $\abs{\Delta_K}^{1/d}$ is bounded (independently of $k$) and all $p_i$ split completely in $K$. If $(p_i)$ splits as $\fp_{i,1}\fp_{i,2} \cdots \fp_{i,d}$, then the congruence $x^2 + y^2 \equiv 0 \pmod{p_i}$ over $\OK$ is equivalent to $(x+\iota y)(x-\iota y) \equiv 0 \pmod {\fp_{i,j}}$ for all $j$, where $\iota$ is again a square root of $-1$. A similar argument then yields that for $a, b \in A$, the squared distance $\abs{a-b}^2$ is divisible by $Q$ roughly $2^{kd}/Q^d$ of the time. Therefore we should expect $\mu(A) \gsim 2^{kd} \abs{A}^2/n$, which yields a polynomial improvement as $n$ is exponential in $d$. Working with the number field incurs losses, but such losses are polynomial in the discriminant, and by choosing $k$ to be sufficiently large at the start the term $2^{kd}$ overpowers any power of $\abs{\Delta_K}$.

\subsection*{Acknowledgments}
We would like to thank Ernie Croot, David Conlon, Junzhe Mao, Oliver Roche-Newton, Will Sawin, Adam Sheffer, Jozsef Solymosi, and Kyle Yip for helpful discussions. 

We would also like to acknowledge the usage of AI in preparation of this manuscript. The collaboration and present paper started with the first author's manuscript from \cite{LeeGithub}, who, independently of the other authors, used ChatGPT to discover the existence of a set of $n$ points $P$ in $\mathbb{R}^{2}$ with the property that every subset \(A\subseteq P\) of size \(|A|\ge n^{1/2-\delta}\) determines a repeated distance. This established that $\fdist(n)\lsim n^{1/2-\delta}$. The proof from \cite{LeeGithub} combined the new local-to-global combinatorial sieve from \cite{CMPSY2026} with the number field towers from \cite{OpenAIUnitDistanceBlog}. This general idea was first introduced by the second author in the recent counterexample for the Elekes-R\'onyai problem \cite{PohoataER}, and subsequently also discussed in the concluding remarks from \cite{CMPSY2026}, which claimed that the ideas from \cite{PohoataER} can be used to show that $\fdist(n)\lsim n^{1/2-\delta}$, as well as $\fiso(n)\lsim n^{1-\delta}$.

The fact that one can establish both of these results through a single construction, which also serves as a counterexample for the unit distance conjecture (and with a simpler overall analysis), is entirely new and was the main reason for writing this separate joint note.

\section{Algebraic Number Theory}
\label{sec:arithmetic-input}

For a number field \(K\), let \(\Disc_K\) denote its
discriminant (which is positive if $K$ is totally real) and let $\rd(K)=\abs{\Disc_K}^{1/[K:\Q]}$ denote its root discriminant.

We will use the following input as a black box, which is a mild strengthening of \cite[Proposition~2.3]{UnitDistanceRemarks}.

\begin{proposition}
\label[proposition]{prop:tower}
There exists an infinite tower of totally real fields
$
        \Q=K_0\subset K_1\subset K_2\subset\cdots
$
such that
\begin{itemize}
    \item For all $i \geq 0$, we have $[K_i : \setq] = 2^i$.
    \item All root discriminants $\rd(K_i)$ are bounded by some constant $D$.
    \item There is an infinite set $\cP$ of $1$ mod $4$ (rational) primes splitting completely in every $K_i$.
\end{itemize}
\end{proposition}

\begin{proof}
Applying a theorem of Hajir-Maire-Ramakrishna \cite[Theorem~4]{HajirMaireRamakrishna2021} with $K=\Q$, $p=2$, and $S=\{3,5,7,11,13,17,\infty\}$ yields an infinite totally real Galois pro-\(2\) extension
\(L/\Q\) of bounded root discriminant in which infinitely many rational primes
split completely.

As noted in the proof of \cite[Proposition~2.3]{UnitDistanceRemarks}, a simple modification of this argument allows us to further impose the $1$ mod $4$ condition. Namely, the proof of \cite[Theorem~4]{HajirMaireRamakrishna2021} selects $\cP$ by selecting distinct primes whose Frobenius elements fall in an infinite sequence of open normal subgroups of an infinite Galois group $\Gal(\tilde \setq/\setq)$. However, every open normal subgroup $H \trianglelefteq \Gal(\tilde \setq/\setq)$ contains infinitely many Frobenius elements of $1$ mod $4$ primes, by applying the Chebotarev density theorem to $\tilde \setq^H(i)$.

We may get the $K_i$ from $L$ as every infinite pro-$p$ $G$ admits a filtration $G = G_0 \triangleright G_1 \triangleright \cdots$ where $[G_i : G_{i+1}] = p$.
\end{proof}

Note that unlike \cite{UnitDistanceRemarks}, we need the root discriminant to be bounded independently of the number of completely split primes.

Henceforth, fix a field \(K=K_i\) from the tower and put $d=[K:\Q]=2^i$. Write its real embeddings as $\sigma_1,\ldots,\sigma_d\colon K\hookrightarrow\R$ and let $\Sigma(\alpha)=(\sigma_1(\alpha),\ldots,\sigma_d(\alpha))\in\R^d$ be the Minkowski embedding.  The lattice \(\Sigma(\OK)\) has covolume
\(\Disc_K^{1/2}\).  More generally, if \(I\subseteq\OK\) is a nonzero (integral)
ideal, then \(\Sigma(I)\) has covolume $\Norm(I)\Disc_K^{1/2}$, where $\Norm(I)=|\OK/I|$ (see e.g.\ \cite[Proposition~4.26]{Milne2020}).

For \(X\ge0\), define the symmetric Minkowski box
\[
        B_K(X)
        \coloneq
        \{\alpha\in\OK:
        |\sigma_j(\alpha)|\le X\text{ for every }1\le j\le d\}.
\]
As in \cite{UnitDistanceRemarks,BSSZ2026,PohoataER}, we will need some lattice estimates.
\begin{lemma}
\label{lem:minkowski-box}
For every \(X\ge0\),
\[
        X^d\Disc_K^{-1/2}
        \le
        |B_K(X)|.
\]
Moreover, if $\fa \subseteq \OK$ is a nonzero (integral) ideal,
\[|\fa\cap B_K(X)|\le
\left(1+\frac{2X}{\Norm(\fa)^{1/d}}\right)^d.\]
\end{lemma}

\begin{proof}
The first inequality is exactly \cite[Lemma~3.3]{BSSZ2026} and
\cite[Lemma~3.2]{PohoataER}. To prove the second (a strengthening of \cite[Lemma~3.3]{PohoataER}), we recall that for all nonzero $a \in \fa$, the norm $\abs{\Norm_{K/\Q}(a)} = \Norm((a))$ is a multiple of $\Norm(\fa)$ and hence at least $\Norm(\fa)$. Thus there is some $i$ such that $\abs{\sigma_i(a)} \geq \Norm(\fa)^{1/d}$, meaning that the points of $\Sigma(\fa)$ are separated by $\Norm(\fa)^{1/d}$ in the $\ell^\infty$ norm.

Around every point of \(\Sigma(\fa)\cap[-X,X]^d\), place a half-open cube of
side length \(\Norm(\fa)^{1/d}\).  These cubes are pairwise disjoint and their union lies
in $[-X-\Norm(\fa)^{1/d}/2,X+\Norm(\fa)^{1/d}/2]^d$. Comparing volumes gives
\[
        \#\bigl(\Sigma(\fa)\cap[-X,X]^d\bigr)\,\Norm(\fa)
        \le
        (2X+\Norm(\fa)^{1/d})^d,
\]
and rearranging yields the desired.
\end{proof}

\section{Proof of \texorpdfstring{\cref{thm:robust-ramanujan}}{Theorem \ref{thm:robust-ramanujan}}}
\label{sec:master-sieve}
Let
\[q((a_1,a_2),(b_1,b_2)) = (a_1-b_1)^2+(a_2-b_2)^2\]
and for a finite set $A \subset K^2$, let
\[\mu(A) \coloneq \max_{\lambda \in K}
        \#\{(a,b)\in A^2:a\neq b,\ 
        q(a,b)=\lambda\}.\]
Note that since the embeddings $\sigma_i$ are injective, given any $A \subset K^2$, mapping it to $\setr^2$ via any embedding $\sigma_i$ yields a set $A'$ with $\mu(A) = \mu(A')$.

\begin{lemma} \label{lem:sieve}
Let $p_1,p_2,\ldots,p_k \in \cP$ be distinct (rational) primes and let $Q = p_1p_2 \cdots p_k$. Suppose $X \geq Q$ and let $A \subseteq B_K(X)^2$ be a set with $\abs{A} \geq 2Q^d$. Then
\[\mu(A) \geq \abs{A}^2 \paren*{\frac{c}{X^2} \prod_{i=1}^k \frac{2p_i}{p_i+1}}^d\]
for some absolute constant $c > 0$.
\end{lemma}
\begin{proof}
For $i \in [k]$, let \(\iota_i\in\F_{p_i}\) be a square root of $-1$ and let the factorization of $(p_i)$ be $\fp_{i,1}\cdots \fp_{i,d}$.

For each sign vector
\[
        \eps=(\eps_{i,j})_{
        \substack{1\le i\le k\\1\le j\le d}}
        \in\{\pm1\}^{kd},
\]
define the additive subgroup
\[
        L_\eps
        =
        \{(x,y)\in\OK^2:
        x\equiv \eps_{i,j}\iota_i y
        \pmod{\fp_{i,j}}
        \text{ for every }i,j\}.
\]
Note that the Chinese remainder theorem gives $\abs{\OK^2/L_\eps} = Q^d$.

Let $S_\eps$ be the set of pairs $(a,b) \in A^2$ such that $a \neq b$ and $a-b \in L_\eps$. If we let $n_1,\ldots,n_{Q^d}$ be the sizes of the intersections of $A$ with cosets of $L_\eps$, we have
\[\abs{S_\eps} = \sum_{s=1}^{Q^d} n_s(n_s-1) \geq \frac{\abs{A}^2}{Q^d} - \abs{A} = \frac{\abs{A}(\abs{A}-Q^d)}{Q^d} \geq \frac{\abs{A}^2}{2Q^d},\]
where we have used Cauchy-Schwarz, the fact that $\sum_{s=1}^{Q^{d}} n_{s} = |A|$, and the condition $\abs{A} \geq 2Q^d$.

Let $S = \bigcup_\eps S_\eps$. Take some $(a,b) \in S$ and let $(x,y) = a-b$. Since $x \equiv \eps_{i,j} \iota_i y \pmod{\fp_{i,j}}$ for some $\eps_{i,j} \in \set{\pm 1}$, we have $x^2 + y^2 \equiv 0 \pmod{\fp_{i,j}}$. Thus $q(a,b) = x^2 + y^2$ is divisible by $\prod_{i,j} \fp_{i,j} = Q$. Moreover, suppose $x \equiv \eps_{i,j} \iota_i y \pmod{\fp_{i,j}}$ for both choices of $\eps_{i,j} \in \set{\pm 1}$. Then $x \equiv y \equiv 0 \pmod{\fp_{i,j}}$ and thus $x^2 + y^2 \equiv 0 \pmod{\fp_{i,j}^2}$. Therefore, if we let $w(r)$ be the number of $\fp_{i,j}$ such that $\fp_{i,j}^2 \mid r$, there are at most $2^{w(x^2+y^2)}$ choices of $\eps$ such that $(a,b) \in S_\eps$.

For each $a,b \in B_K(X)^2$, we have $q(a,b) \in B_K(8X^2)$, so the map $q$ maps $S$ to $Q\OK \cap B_K(8X^2)$.
By definition, at most $\mu(A)$ elements of $S$ map to the same element. Therefore, we find that
\[\frac{2^{kd}\abs{A}^2}{2Q^d} \leq \sum_\eps \abs{S_\eps} \leq \sum_{(a,b) \in S} 2^{w(q(a,b))} \leq \mu(A) \cdot \sum_{r \in Q\OK \cap B_K(8X^2)} 2^{w(r)}.\]

If $r$ is a multiple of $Q$, the quantity $2^{w(r)}$ is simply the number of common divisors of $r/Q$ and $Q$. Therefore
\[\sum_{r \in Q\OK \cap B_K(8X^2)} 2^{w(r)} = \sum_{\fa\mid Q}\sum_{r \in Q\fa \cap B_K(8X^2)} 1 = \sum_{\fa\mid Q} \abs{Q\fa \cap B_K(8X^2)}.\]
By \cref{lem:minkowski-box}, we find
\[
\abs{Q\fa \cap B_K(8X^2)} \leq \paren*{1 + \frac{16X^2}{Q\Norm(\fa)^{1/d}}}^d.
\]
Since $Q\Norm(\fa)^{1/d} \leq Q^2 \leq X^2$, summing over all $\fa$ yields
\[
\sum_{\fa\mid Q} \abs{Q\fa \cap B_K(8X^2)} \leq \sum_{\fa\mid Q} \paren*{\frac{17X^2}{Q\Norm(\fa)^{1/d}}}^d = \paren*{\frac{17 X^2}{Q}}^d \sum_{\fa \mid Q} \frac{1}{\Norm(\fa)} = \paren*{\frac{17 X^2}{Q} \prod_{i=1}^k \paren*{1 + \frac{1}{p_i}}}^d.
\]
Putting everything together yields
\[\mu(A) \geq \frac{\abs{A}^2}{2} \paren*{\frac{1}{17X^2} \prod_{i=1}^k \frac{2p_i}{p_i+1}}^d,\]
which shows that $c = 1/34$ suffices.
\end{proof}

It remains to choose parameters wisely.
\begin{proof}[Proof of \cref{thm:robust-ramanujan}]
Let $D$ and $c$ be as in \cref{prop:tower} and \cref{lem:sieve}, respectively, and choose distinct primes $p_1,\ldots,p_k \in \cP$ such that
\[c \prod_{i=1}^k \frac{2p_i}{p_i+1} \geq 2D,\]
which is possible since $\frac{2p_i}{p_i+1} \geq \frac{5}{3}$ and $\cP$ is infinite. Let $Q = p_1p_2 \cdots p_k$.

By setting constants appropriately it suffices to consider $n \geq 100Q^2$. Let $d$ be the unique power of two such that $(10Q)^{2d} \leq n < (10Q)^{4d}$, and let $K$ be the field in the tower of degree $d$. Let $X = n^{1/(2d)}\sqrt{D}$. Note that $X \geq 10Q$.

By \cref{lem:minkowski-box}, $\abs{B_K(X)} \geq X^d D^{-d/2} = n^{1/2}$, and we let $P$ be (the real embedding of) an arbitrary $n$-element subset of $B_K(X)^2$. Then \cref{lem:sieve} shows that for any $A \subseteq P$ of size at least $(2Q)^d$, we have $\mu(A) \geq \abs{A}^2 \cdot (2D/X^2)^d = \abs{A}^2 \cdot 2^d/n$. By choosing $\delta_1$ and $\delta_2$ such that $(100Q^2)^{1/2 - \delta_1} = 2Q$ and $(10000Q^4)^{\delta_2} = 2$, we conclude that for every subset $A \subseteq P$ of size at least $n^{1/2 - \delta_1}$, we have $\mu(A) \geq \abs{A}^2 / n^{1-\delta_2}$.

Finally, we set $\delta = \min(2\delta_1, \delta_2)$. If $2 \leq \abs{A} < n^{1/2 - \delta_1}$, we have $\mu(A) \geq 1 > \abs{A}^2/n^{1-2\delta_1} \geq \abs{A}^2/n^{1-\delta}$. If $\abs{A} \geq n^{1/2-\delta_1}$, we have $\mu(A) \geq \abs{A}^2/n^{1-\delta_2} \geq \abs{A}^2/n^{1-\delta}$. This concludes the proof.
\end{proof}

\printbibliography

@Article{BaloghSolymosi2018,
  author   = {Balogh, J\'ozsef and Solymosi, J\'ozsef},
  title    = {On the number of points in general position in the plane},
  doi      = {10.19086/da.4438},
  eid      = {2018:16},
  issn     = {2397-3129},
  url      = {https://doi.org/10.19086/da.4438},
  fjournal = {Discrete Analysis},
  journal  = {Discrete Anal.},
  year     = {2018},
}

@Article{Singer38,
  author   = {Singer, James},
  title    = {A theorem in finite projective geometry and some applications to number theory},
  doi      = {10.1090/S0002-9947-1938-1501951-4},
  number   = {3},
  pages    = {377--385},
  url      = {https://doi.org/10.1090/S0002-9947-1938-1501951-4},
  volume   = {43},
  fjournal = {Transactions of the American Mathematical Society},
  journal  = {Trans. Amer. Math. Soc.},
  mrnumber = {1501951},
  year     = {1938},
}

@Article{ErdosTuran41,
  author   = {Erd\H{o}s, Paul and Tur\'an, P\'al},
  title    = {On a problem of Sidon in additive number theory, and on some related problems},
  doi      = {10.1112/jlms/s1-16.4.212},
  number   = {4},
  pages    = {212--215},
  url      = {https://doi.org/10.1112/jlms/s1-16.4.212},
  volume   = {16},
  fjournal = {Journal of the London Mathematical Society},
  journal  = {J. London Math. Soc.},
  mrnumber = {0004763},
  year     = {1941},
}

@Online{TaoBlog,
  author = {Tao, Terence},
  date   = {2026-07-03},
  title  = {A digestion of unit distance constructions},
  url    = {https://terrytao.wordpress.com/2026/07/03/a-digestion-of-unit-distance-constructions/},
}

@Online{UnitDistanceRemarks,
  author      = {Alon, Noga and Bloom, Thomas F. and Gowers, W. T. and Litt, Daniel and Sawin, Will and Shankar, Arul and Tsimerman, Jacob and Wang, Victor and Wood, Melanie Matchett},
  date        = {2026-05-20},
  title       = {Remarks on the disproof of the unit distance conjecture},
  eprint      = {2605.20695},
  eprintclass = {math.CO},
  eprinttype  = {arXiv},
}

@Article{Beh,
  author     = {Behrend, F. A.},
  title      = {On sets of integers which contain no three terms in arithmetical progression},
  doi        = {10.1073/pnas.32.12.331},
  issn       = {0027-8424},
  pages      = {331--332},
  url        = {https://doi.org/10.1073/pnas.32.12.331},
  volume     = {32},
  fjournal   = {Proceedings of the National Academy of Sciences of the United States of America},
  journal    = {Proc. Nat. Acad. Sci. U.S.A.},
  mrclass    = {10.0X},
  mrnumber   = {18694},
  mrreviewer = {P.\ Erd\H os},
  year       = {1946},
}

@Online{Bloom,
  author    = {Thomas F. Bloom},
  date      = {2026},
  title     = {Erd\H{o}s Problems},
  url       = {https://www.erdosproblems.com},
  shorthand = {ErdPro},
}

@Online{BSSZ2026,
  author      = {Bloom, Thomas F and Sawin, Will and Schildkraut, Carl and Zhelezov, Dmitrii},
  date        = {2026-05-27},
  title       = {The sum-product conjecture is false for real numbers},
  eprint      = {2605.28781},
  eprintclass = {math.NT},
  eprinttype  = {arXiv},
}

@Book{BMP,
  author     = {Brass, Peter and Moser, William and Pach, J\'anos},
  title      = {Research problems in discrete geometry},
  doi        = {10.1007/0-387-29929-7},
  publisher  = {Springer},
  url        = {https://doi.org/10.1007/0-387-29929-7},
  mrclass    = {52-02 (05-02)},
  mrnumber   = {2163782},
  mrreviewer = {W.\ Kuperberg},
  pagecount  = {xii+499},
  year       = {2005},
}

@Online{ChartonEllenbergWagnerWilliamson2024,
  author      = {Charton, François and Ellenberg, Jordan S. and Wagner, Adam Zsolt and Williamson, Geordie},
  date        = {2024-11-01},
  title       = {PatternBoost: Constructions in Mathematics with a Little Help from AI},
  eprint      = {2411.00566},
  eprintclass = {math.CO},
  eprinttype  = {arXiv},
}

@Online{CMPSY2026,
  author      = {Croot, Ernie and Mao, Junzhe and Pohoata, Cosmin and Sheffer, Adam and Yip, Chi Hoi},
  date        = {2026-06-16},
  title       = {A combinatorial large sieve for Sidon sets, distances, and norm forms},
  eprint      = {2606.17487},
  eprintclass = {math.NT},
  eprinttype  = {arXiv},
}

@Article{ConlonFoxGasarchHarrisUlrichZbarsky15,
  author     = {Conlon, David and Fox, Jacob and Gasarch, William and Harris, David G. and Ulrich, Douglas and Zbarsky, Samuel},
  title      = {Distinct volume subsets},
  doi        = {10.1137/140954519},
  issn       = {0895-4801,1095-7146},
  number     = {1},
  pages      = {472--480},
  url        = {https://doi.org/10.1137/140954519},
  volume     = {29},
  fjournal   = {SIAM Journal on Discrete Mathematics},
  journal    = {SIAM J. Discrete Math.},
  mrclass    = {52C10},
  mrnumber   = {3319845},
  mrreviewer = {Seyed\ Amin\ Seyed Fakhari},
  year       = {2015},
}

@Article{Erdos46,
  author     = {Erd\"os, P.},
  title      = {On sets of distances of {$n$} points},
  doi        = {10.2307/2305092},
  issn       = {0002-9890,1930-0972},
  pages      = {248--250},
  url        = {https://doi.org/10.2307/2305092},
  volume     = {53},
  fjournal   = {American Mathematical Monthly},
  journal    = {Amer. Math. Monthly},
  mrclass    = {48.0X},
  mrnumber   = {15796},
  mrreviewer = {I.\ Kaplansky},
  year       = {1946},
}

@Article{Erdos1980Survey,
  author     = {Erd\H os, Paul},
  title      = {A survey of problems in combinatorial number theory},
  pages      = {89--115},
  url        = {https://renyi.hu/~p_erdos/1980-03.pdf},
  volume     = {6},
  fjournal   = {Annals of Discrete Mathematics},
  journal    = {Ann. Discrete Math.},
  mrclass    = {10-XX (05Axx)},
  mrnumber   = {593525},
  mrreviewer = {A.\ L.\ Whiteman},
  year       = {1980},
}

@Article{ErGu70,
  author     = {Erd\H os, P. and Guy, R. K.},
  title      = {Distinct distances between lattice points},
  issn       = {0013-6018},
  pages      = {121--123},
  url        = {https://combinatorica.hu/~p_erdos/1970-03.pdf},
  volume     = {25},
  fjournal   = {Elemente der Mathematik. Revue de Math\'ematiques \'El\'ementaires. Rivista de Matematica Elementare},
  journal    = {Elem. Math.},
  mrclass    = {10.25},
  mrnumber   = {281691},
  mrreviewer = {L.\ Beran},
  year       = {1970},
}

@Article{HajirMaireRamakrishna2021,
  author     = {Hajir, Farshid and Maire, Christian and Ramakrishna, Ravi},
  title      = {Cutting towers of number fields},
  doi        = {10.1007/s40316-021-00156-8},
  issn       = {2195-4755,2195-4763},
  number     = {2},
  pages      = {321--345},
  url        = {https://doi.org/10.1007/s40316-021-00156-8},
  volume     = {45},
  fjournal   = {Annales Math\'ematiques du Qu\'ebec},
  journal    = {Ann. Math. Qu\'e.},
  mrclass    = {11R29 (11R21 11R37)},
  mrnumber   = {4308183},
  mrreviewer = {Abdelmalek\ Azizi},
  year       = {2021},
}

@Article{KSS,
  author     = {Koml\'os, J. and Sulyok, M. and Szemeredi, E.},
  title      = {Linear problems in combinatorial number theory},
  doi        = {10.1007/BF01895954},
  issn       = {0001-5954,1588-2632},
  pages      = {113--121},
  url        = {https://doi.org/10.1007/BF01895954},
  volume     = {26},
  fjournal   = {Acta Mathematica. Academiae Scientiarum Hungaricae},
  journal    = {Acta Math. Acad. Sci. Hungar.},
  mrclass    = {10A99},
  mrnumber   = {364087},
  mrreviewer = {S.\ L. G. Choi},
  year       = {1975},
}

@Online{LeeGithub,
  author = {Sungchul Lee},
  date   = {2026},
  title  = {Planar point sets with small distance-distinct subsets},
  url    = {https://github.com/lsngchl/Erdos1208},
  note   = {Preprint},
}

@Article{LefmannThiele95,
  author     = {Lefmann, Hanno and Thiele, Torsten},
  title      = {Point sets with distinct distances},
  doi        = {10.1007/BF01299744},
  issn       = {0209-9683},
  number     = {3},
  pages      = {379--408},
  url        = {https://doi.org/10.1007/BF01299744},
  volume     = {15},
  fjournal   = {Combinatorica. An International Journal on Combinatorics and the Theory of Computing},
  journal    = {Combinatorica},
  mrclass    = {52C10},
  mrnumber   = {1357284},
  mrreviewer = {Konrad\ J.\ Swanepoel},
  year       = {1995},
}

@Online{Milne2020,
  author = {J. S. Milne},
  date   = {2020},
  title  = {Algebraic Number Theory},
  url    = {https://www.jmilne.org/math/CourseNotes/ANT.pdf},
}

@Online{OpenAIUnitDistanceBlog,
  author = {OpenAI},
  date   = {2026},
  title  = {An OpenAI model has disproved a central conjecture in discrete geometry},
  url    = {https://openai.com/index/model-disproves-discrete-geometry-conjecture/},
}

@Online{PohoataER,
  author      = {Pohoata, Cosmin},
  date        = {2026-06-11},
  title       = {Split primes and the Elekes-Rónyai problem},
  eprint      = {2606.13619},
  eprintclass = {math.NT},
  eprinttype  = {arXiv},
}

@Article{Ramanujan1915,
  author   = {Ramanujan, S.},
  title    = {Highly composite numbers},
  doi      = {10.1112/plms/s2\_14.1.347},
  issn     = {0024-6115},
  pages    = {347--409},
  series   = {2},
  url      = {https://doi.org/10.1112/plms/s2_14.1.347},
  volume   = {14},
  fjournal = {Proceedings of the London Mathematical Society. Second Series},
  journal  = {Proc. London Math. Soc.},
  mrclass  = {11A25 (11N64)},
  mrnumber = {4796923},
  year     = {1915},
}

@Article{Roth,
  author     = {Roth, K. F.},
  title      = {On certain sets of integers},
  doi        = {10.1112/jlms/s1-28.1.104},
  issn       = {0024-6107,1469-7750},
  pages      = {104--109},
  url        = {https://doi.org/10.1112/jlms/s1-28.1.104},
  volume     = {28},
  fjournal   = {The Journal of the London Mathematical Society},
  journal    = {J. London Math. Soc.},
  mrclass    = {10.0X},
  mrnumber   = {51853},
  mrreviewer = {P.\ Erd\H os},
  year       = {1953},
}

@Online{Sawin2026,
  author      = {Sawin, Will},
  date        = {2026-05-20},
  title       = {An explicit lower bound for the unit distance problem},
  eprint      = {2605.20579},
  eprintclass = {math.CO},
  eprinttype  = {arXiv},
}

@Online{TaoOptimization84a,
  author    = {Davis, Damek and Ivanisvili, Paata and Tao, Terence and contributors},
  date      = {2026},
  title     = {Optimization Constants in Mathematics},
  url       = {https://teorth.github.io/optimizationproblems},
  shorthand = {OptCon},
}

@Article{Erd57,
  author     = {Erd\H os, P\'al},
  title      = {Néhány geometriai problémáról},
  issn       = {0025-519X},
  pages      = {86--92},
  url        = {https://www.renyi.hu/~p_erdos/1957-03.pdf},
  volume     = {8},
  fjournal   = {Matematikai Lapok. Bolyai J\'anos Matematikai T\'arsulat},
  journal    = {Mat. Lapok},
  mrclass    = {50.00},
  mrnumber   = {99617},
  mrreviewer = {J.\ Acz\'el},
  year       = {1957},
}

@Online{Sheffer2014,
  author      = {Sheffer, Adam},
  date        = {2014-06-08},
  title       = {Distinct Distances: Open Problems and Current Bounds},
  eprint      = {1406.1949},
  eprintclass = {math.CO},
  eprinttype  = {arXiv},
}
\end{document}